\documentclass[oneside,leqno,11pt]{amsart}

\title[Samuel Coskey: Profinite actions of $\SL_n(\ZZ)$]{Borel
  reductions of profinite actions of $\SL_n(\ZZ)$}
\author{Samuel Coskey}
\address{The Graduate Center of The City University of
  New York, Mathematics Program, 365 Fifth Avenue, New York, NY 10016}
\email{scoskey@nylogic.org}
\urladdr{http://math.rutgers.edu/$\mathord\sim$scoskey}

\usepackage[marginratio=1:1]{geometry}
\usepackage{setspace}
\usepackage{mathpazo,bm}
\usepackage{eucal}
\usepackage{mydefs}

\begin{document}
\begin{abstract}
  Greg Hjorth and Simon Thomas proved that the classification problem
  for torsion-free abelian groups of finite rank \emph{strictly
    increases} in complexity with the rank.  Subsequently, Thomas
  proved that the complexity of the classification problems for
  $p$-local torsion-free abelian groups of fixed rank $n$ are
  \emph{pairwise incomparable} as $p$ varies.  We prove that if $3\leq
  m<n$ and $p,q$ are distinct primes, then the complexity of the
  classification problem for $p$-local torsion-free abelian groups of
  rank $m$ is again incomparable with that for $q$-local torsion-free
  abelian groups of rank $n$.
\end{abstract}
\maketitle

\section{Introduction}

This paper follows upon the methods introduced in \cite{hjorth} and
\cite{adamskechris}, and further specialized in \cite{super},
\cite{plocal}, and \cite{quasi}.  The theme of these papers is the
intersection of two related pursuits:
\begin{itemize}
\item the study of the general structure of the countable Borel
  equivalence relations, and
\item the particular case of the complexity of the classification
  problem for torsion-free abelian groups of finite rank.
\end{itemize}
At the heart of each is the use of powerful methods from
ergodic theory and the superrigidity theory of Lie groups.

The study of Borel equivalence relations begins with the observation
that many classification problems can be identified with an
equivalence relation on a standard Borel space (\emph{i.e.}, a Polish
space equipped just with its $\sigma$-algebra of Borel sets).  For
instance, each group with domain $\NN$ is determined by its group
operation, a subset of $\NN^3$.  Hence, the space of countable groups
may be identified with a subset $X_{\mathcal G}\subset\PP(\NN^3)$.
Studying the classification problem for countable groups thus amounts
to studying the isomorphism equivalence relation $\oiso_{\mathcal G}$
on $X_{\mathcal G}$.  The relation $\oiso_{\mathcal G}$ is extremely
complex in the intuitive sense that to check whether
$(\NN;\ord\times_1)\iso_{\mathcal G}(\NN;\ord\times_2)$, one must
conduct an unbounded search for a witnessing bijection
$\phi\from\NN\into\NN$.  This intuition is reflected in descriptive
set theory in part by the fact that $\oiso_{\mathcal G}$ is not a
Borel subset of $X_{\mathcal G}\times X_{\mathcal G}$.

However, there are many subcollections of the class of countable
groups whose isomorphism equivalence relation \emph{is} Borel.  For
instance, in this paper we will focus on the space of torsion-free
abelian groups of finite rank.  Since any torsion-free abelian group
of rank $n$ is isomorphic to a subgroup of $\QQ^n$, the space of
torsion-free abelian groups of rank $n$ can be identified with a
subset $R(n)\subset\mathcal P(\QQ^n)$.  Moreover, it is easily seen
that for $A,B\leq\QQ^n$, we have that $A\iso B$ iff there exists
$g\in\GL_n(\QQ)$ such that $B=g(A)$.  It follows easily that the
isomorphism equivalence relation $\oiso_n$ on $R(n)$ is a Borel
equivalence relation.

The Borel/non-Borel dichotomy is a useful one, but we will shortly
introduce a much finer notion of complexity which is specially
tailored for equivalence relations.  As a start, an equivalence
relation $E$ on the standard Borel space $X$ is said to be
\emph{smooth}, or completely classifiable, if there exists a standard
Borel space $Y$ and a Borel function $f\from X\into Y$ satisfying
\[x\mathrel{E}x'\iff f(x)=f(x')\;.
\]
In other words, $Y$ is a space of complete invariants for the
classification problem up to $E$.  The condition that $f$ is Borel
amounts to the requirement that the invariants can be computed in a
reasonably ``explicit'' manner.  For instance, the classification
problem for countable divisible groups is smooth.  Indeed, any
countable divisible group is decomposable into a product of Pr\"ufer
$p$-groups, and so any such group $A$ is determined up to isomorphism
by the sequence that lists the number of factors of each Pr\"ufer
group in a decomposition of $A$.

On the other hand, it follows from a 1937 result of Baer that even the
classification problem for torsion-free abelian groups of rank $1$ is
not smooth.  To explain this, however, we must first define the notion
of \emph{Borel reducibility}.  If $E,F$ are equivalence relations on
the standard Borel spaces $X,Y$, then we say $E$ is Borel reducible to
$F$ and write $E\leq_BF$ iff there exists a Borel function
$f\from X\into Y$ satisfying
\[x\mathrel{E}x'\iff f(x)\mathrel{F}f(x')\;.
\]
We then let $E\sim_BF$ iff $E\leq_BF$ and $F\leq_BE$, $E\perp_BF$ iff
$E\not\leq_BF$ and $F\not\leq_BE$, and finally $E<_BF$ iff $E\leq_BF$
and $E\not\sim_BF$.  In these terms, Baer's result implies that
$\oiso_1\sim_BE_0$, where $E_0$ is the equivalence relation defined on
$2^\NN$ by $x\mathrel{E}_0y$ iff $x(n)=y(n)$ for all but finitely many
$n$.  It is an elementary fact that $E_0$ is nonsmooth (in fact it is
the $\leq_B$-least nonsmooth Borel equivalence relation), and so it
follows that $\oiso_1$ is nonsmooth as well.

For a span of 60 years following Baer's result, the classification
problem for torsion-free abelian groups of rank $2$ and higher
remained open.  Although Kurosh and Malcev wrote down complete
invariants for torsion-free abelian groups of rank $2$, they were
considered inadequate as a solution to the classification problem
because it was as difficult to distinguish the invariants as it was
the groups themselves.  In 1998, Hjorth proved in \cite{hjorth} that
$E_0<_B\oiso_2$, and hence that the classification problem for
torsion-free abelian groups of rank $2$ is indeed strictly more
complicated than that for rank $1$.  Hjorth's solution did not provide
any method for dealing with the torsion-free abelian groups of rank
greater than $2$.  In particular, it remained open whether $\oiso_2$
is universal for all torsion-free abelian groups of finite rank, and
if it's not, then whether $\oiso_3$ is universal, and so on.

This question was of major interest since the $\oiso_n$ are examples
\emph{countable Borel equivalence relations}, and it was unknown at
the time whether there could be an infinite strictly ascending chain
of countable Borel equivalence relations.  Here, a Borel equivalence
relation $E$ is said to be \emph{countable} iff every $E$-class is
countable.  For instance, let $\Gamma$ be a countable group and
suppose that $\Gamma$ acts in a Borel fashion on the standard Borel
space $X$.  Then the induced \emph{orbit equivalence relation}
$E_\Gamma$, defined on $X$ by
\[x\mathrel{E}_\Gamma y\iff\Gamma x=\Gamma y\;,
\]
is clearly countable and easily seen to be Borel.  For instance, by
our earlier remarks concerning the space $R(n)$ of torsion-free
abelian groups of rank $n$, we have that the isomorphism relation
$\oiso_n$ is exactly the orbit equivalence relation on $R(n)$ induced
by the action of $\GL_n(\QQ)$.  By an amazing result of Feldman and
Moore \cite{feldmanmoore}, \emph{every} countable Borel equivalence
relation arises as the orbit equivalence relation induced by a Borel
action of some countable group.

Returning to Hjorth's question of whether $\oiso_3$ is more complex
than $\oiso_2$, the first progress was made by Adams and Kechris in
\cite{adamskechris}, who answered the analogous question for the class
of rigid groups.  Here, a group $A$ is said to be \emph{rigid} iff its
only automorphisms are $\pm Id$.  Let $S(n)\subset R(n)$ denote the
subset consisting of just the rigid torsion-free abelian groups of
rank $n$, and let $\oiso_n^*$ be the restriction of the isomorphism
equivalence relation to $S(n)$.  Adams and Kechris proved the
following:

\begin{thm*}[\protect{\cite[Theorem~6.1]{adamskechris}}]
  For all $n$, we have $\oiso_n^*<_B\oiso_{n+1}^*$.
\end{thm*}

This was one of the earliest results in the subject which separated
two known equivalence relations; indeed, before this result there were
only six known countable Borel equivalence relations up to Borel
bireducibility.  The proof made use of some powerful results from the
ergodic theory of lattices in Lie groups, most notably, Zimmer's
cocycle superrigidity theorem.  The reader who is familiar with
Zimmer's theorem may wonder exactly how it is relevant to this
problem.  But recall that $\oiso_n^*$ is induced by the action of
$\GL_n(\QQ)$ on $S(n)$, and note the following two facts:
\begin{itemize}
\item There exists an ergodic, $\SL_n(\ZZ)$-invariant probability
  measure on $S(n)$ (see \cite{hjorth} or
  \cite[Theorem~2.4]{thomas_survey}), and
\item $\SL_n(\ZZ)$ is a lattice in the higher-rank simple Lie group
  $\SL_n(\RR)$ (see \cite[Theorem~3.1.7]{zimmer}).
\end{itemize}
Of course, more is necessary to meet the hypotheses Zimmer's theorem,
and even then Adams and Kechris expended a great deal of effort to
extract information from its conclusion.  Shortly after this was done,
Thomas was able to refine in \cite{torsionfree} the method of Adams
and Kechris to fully answer the question on the complexity of the
isomorphism problem for torsion-free abelian groups of rank $3$ and
higher.

\begin{thm*}[\protect{\cite[Theorem~1.4]{thomas_survey}}]
  For all $n$, we have $\oiso_n<_B\oiso_{n+1}$.
\end{thm*}

As a stepping stone towards this result, Thomas proved the analogous
result for the quasi-isomorphism problem.  Here, we say that subgroups
$A,B\leq\QQ^n$ are \emph{quasi-isomorphic} iff $B$ is commensurable
with an isomorphic copy of $A$.  Let $\oqiso_n$ denote the
quasi-isomorphism equivalence relation on the space $R(n)$ of
torsion-free abelian groups of rank $n$.

\begin{thm*}[\protect{\cite[Theorem~4.6]{thomas_survey}}]
  For all $n$, we have $\oqiso_n<_B\oqiso_{n+1}$.
\end{thm*}

These results of Adams-Kechris and Thomas provided the first examples
of infinite chains of \emph{naturally occurring} classification
problems.  The proofs again made use of Zimmer's cocycle superrigidity
theorem for lattices in higher rank Lie groups.  Very loosely
speaking, at the heart of the proof that
$\oiso_{n+1}\not\leq_B\oiso_n$ is the simple observation that the
``dimension'' of $\SL_{n+1}(\ZZ)$ is larger than that of $\SL_n(\ZZ)$
(or more precisely, the rank of the ambient Lie group $\SL_{n+1}(\RR)$
is larger than that of $\SL_n(\RR)$).

Thomas later gave an example of an infinite \emph{antichain} of
naturally occurring equivalence relations.  Recall that a torsion-free
abelian group $A$ is said to be \emph{$p$-local} iff it is
$q$-divisible for every prime $q\neq p$.  Let $\oiso_{n,p}$ denote the
isomorphism equivalence relation (and $\oqiso_{n,p}$ the
quasi-isomorphism relation) on the space of $p$-local torsion-free
abelian groups of rank $n$.  Thomas proved the following:

\begin{thm*}[\protect{\cite[Theorem~1.2 and implicit]{plocal}}]
  Let $p,q$ be distinct primes and $n\geq3$.  Then we have:
  \begin{itemize}
  \item $\oiso_{n,p}\perp_B\oiso_{n,q}$, and
  \item $\oqiso_{n,p}\perp_B\oqiso_{n,q}$.
  \end{itemize}
\end{thm*}

Before this theorem, every Borel non-reducibility result in the area
of torsion-free abelian groups had relied on some notion of the
dimension of (the ambient Lie group of) the acting group as an
invariant.  The significance of this result is that this dimension is
fixed, since of course both $\oqiso_{n,p}$ and $\oqiso_{n,q}$ are
induced by actions of the same group.

This left open the question of whether the locality prime $p$ could be
used to distinguish between isomorphism relations when the dimension
is \emph{not} fixed.

\begin{thm*}
  Let $p,q$ be distinct primes and $m,n\geq3$.  Then we have:
  \begin{enumerate}
  \item[\textbf{A.}] $\oiso_{m,p}\perp_B\oiso_{n,q}$, and
  \item[\textbf{B.}] $\oqiso_{m,p}\perp_B\oqiso_{n,q}$.
  \end{enumerate}
\end{thm*}

More generally, one might ask what role the dimension plays in
deciding whether $E\leq_BF$.  Theorems~A and B shed some light on this
question, since in these cases the dimension has no effect so long as
it is greater than $2$.  Theorem~A will be established in Corollary
\ref{mainthmp1}, and Theorem~B in Corollary \ref{mainthmp2}.  These
results unfortunately leave open a slightly more technical question,
based on the following result from \cite{quasi}.

\begin{thm*}[\protect{\cite[Theorem~B]{quasi}}]
  If $n\geq3$, then $\oiso_{n,p}$ is Borel incomparable with
  $\oqiso_{n,p}$.
\end{thm*}

It would be extremely interesting to know if the
isomorphism/quasi-isomorphism distinction is sufficient to establish
Borel incomparability between the two classification problems, again
even as the dimension increases.

\begin{conj*}
  For $m,n\geq3$ and $p,q$ prime, we have
  $\oiso_{m,p}\perp_B\oqiso_{n,q}$.
\end{conj*}

The substantial case is when $q=p$, since if $q\neq p$ then this can
easily be shown using the methods in this paper.

The remainder of this paper is organized as follows.  In the next
section, we discuss some properties of the action of a dense subgroup
of a compact group $K$ on homogeneous $K$-spaces.  We then state and
prove a result due to Furman which implies that these actions exhibit
some intrinsic rigidity.  We shall pay particular attention to the
Grassmann space consisting of all linear subspaces of $\QQ_p^n$
together with its $\SL_n(\ZZ)$-action.  In the third section we shall
state a superrigidity result of Ioana, and use it to establish the
Borel incomparability of some natural equivalence relations on
Grassmann space.  In the last section, we explain how the isomorphism
and quasi-isomorphism equivalence relations can be viewed as
equivalence relations on Grassmann spaces, and use this together with
the results of Section~3 to prove Theorems~A and B.

I would like to acknowledge Simon Thomas for pointing out this line of
research, Scott Schneider for helpful conversations on this
subject, and the referee for pointing out several inaccuracies.

\section{Homogeneous spaces of compact groups}

In this section, we give an introduction to homogeneous spaces of
compact groups and affine maps between them.  We then give the
definition of ergodicity of a general measure-preserving action, and a
characterization of ergodicity in the case of homogeneous spaces.
Finally, we present two lemmas (due to Gefter and Furman), which
loosely speaking imply that if $\Gamma,\Lambda$ act ergodically on
homogeneous spaces, then any conjugacy between these actions comes
from an affine map.

If $K$ is a compact group, then a \emph{homogeneous} $K$-space is a
standard Borel space $X$ together with a transitive Borel action of
$K$ on $X$.  If $X$ is a homogeneous $K$-space, then $X$ is isomorphic
as a $K$-space to the left coset space $K/L$, where $L\leq K$ is the
stabilizer of an arbitrary point $x\in X$.  Hence, $X$ admits a
$K$-invariant \emph{Haar measure}, namely the push-forward to $X$ of
the usual Haar measure on $K$.

For instance, let $\Gr_k(\QQ_p^m)$ denote the \emph{Grassmann space}
of all $k$-dimensional subspaces of $\QQ_p^m$.  By
\cite[Proposition~6.1]{super}, the compact group $\SL_m(\ZZ_p)$ acts
transitively on $\Gr_k(\QQ_p^m)$, and it follows that $\Gr_k(\QQ_p^m)$
is a homogeneous $\SL_m(\ZZ_p)$-space.  For purely \ae sthetic
reasons, we sometimes denote $\Gr_1(\QQ_p^m)$ instead by
$\PP(\QQ_p^m)$.

\begin{defn}
  For $i=0,1$, let $K_i$ be a compact group and $L_i$ a closed
  subgroup.  A map $f\from K_0/L_0\into K_1/L_1$ between homogeneous
  spaces is said to be \emph{affine} iff there exists a homomorphism
  $\Phi\from K_0\into K_1$ and $t\in K_1$ such that
  $f(kL_0)=\Phi(k)tL_1$ for almost all $k\in K_0$.
\end{defn}

Affine maps are the natural morphisms between homogeneous spaces of
compact groups.  It is trivial to see that any affine map
$f(kL_0)=\Phi(k)tL_1$ has the property that the pair $(\Phi,f)$ is a
\emph{homomorphism of permutation groups}, in the sense that
$f(kx)=\Phi(k)f(x)$ for all $x\in K_0/L_0$.  Lemmas~\ref{lem_affine}
and \ref{lem_extends}, taken together, provide a very strong converse
to this observation.  First, we shall need to introduce the notion of
ergodicity of a measure-preserving action.

Let $\Gamma$ be a countable group acting on the standard Borel space
$X$ (which we denote by $\Gamma\actson X$), and suppose the action
preserves a probability measure on $X$.  Then the action
$\Gamma\actson X$ is said to be \emph{ergodic} iff every
$\Gamma$-invariant measurable subset $A\subset X$ has either
$\mu(A)=0$ or $\mu(A)=1$.  We shall have more use for the following
equivalent formulation of this property: $\Gamma\actson X$ is ergodic
iff for every standard Borel space $Y$ and every $\Gamma$-invariant
Borel function $\beta\from X\into Y$, we have that $\beta$ is
constant on a conull set.

For instance, if $X$ is a homogeneous $K$-space and $\Gamma$ is a
countable subgroup of $K$, then $\Gamma$ acts on $X$ and preserves the
Haar measure.  It is easily seen that in this case, $\Gamma\actson X$
is ergodic iff $\Gamma$ is dense in $K$.  The statement and proof of
Lemma~\ref{lem_affine} were extracted from \cite[Theorem~3.3]{gefter}.

\begin{lem}
  \label{lem_affine}
  For $i=0,1$ let $K_i/L_i$ be a homogeneous space for the compact
  group $K_i$, let $\Gamma_i<K_i$ be a countable dense subgroup, and
  suppose that
  \[(\phi,f)\from\Gamma_0\actson K_0/L_0\longrightarrow\Gamma_1\actson K_1/L_1
  \]
  is a homomorphism of permutation groups.  If $\phi$ extends to a
  homomorphism $\Phi\from K_0\into K_1$, then after adjusting $f$ on
  a set of measure zero, $f$ is an affine map.
\end{lem}

\begin{proof}
  Following Gefter's argument, define the map $\beta\from K_0\into
  K_1/L_1$ by
  \[\beta(k)\defeq\Phi(k)^{-1}f(kL_0)\;.
  \]
  We first observe that $\beta$ is $\Gamma_0$-invariant.  Indeed, for
  $\gamma\in\Gamma_0$, we compute that:
  \begin{align*}
    \beta(\gamma k)&=\Phi(\gamma k)^{-1}f(\gamma kL_0)\\
    &=\Phi(k)^{-1}\Phi(\gamma)^{-1}\phi(\gamma)f(kL_0)\\
    &=\Phi(k)^{-1}f(kL_0)\\
    &=\beta(k)\;.
  \end{align*}
  Now, since $\Gamma_0$ is a dense subgroup of $K_0$, the action
  $\Gamma_0\actson K_0$ is ergodic.  Hence, there exists $t\in K_1$
  such that for almost every $k\in K_0$, we have that $\beta(k)=tL_1$.
  In other words, there exists a conull subset $K_0^*\subset K_0$ such
  that for all $k\in K_0^*$ we have the identity
  $f(kL_0)=\Phi(k)tL_1$.  Now, we will be done if we show that the
  function $f'(kL_0)\defeq\Phi(k)tL_1$ is well-defined, for then $f'$
  is an affine map which is equal to $f$ almost everywhere.

  For this, a moment's pause reveals that $f'$ is well-defined if and
  only if $\Phi(L_0)=tL_1t^{-1}$.  Now, given $\ell\in L_0$, choose
  $k\in K_0^*$ such that also $k\ell\in K_0^*$.  (This is possible:
  the right Haar measure has the same null sets as the left Haar
  measure, so $K_0^*\ell^{-1}$ is non-null.)  We now have:
  \begin{align*}
    \Phi(k)tL_1&=f(kL_0)\\
    &=f(k\ell L_0)\\
    &=\Phi(k\ell)tL_1\\
    &=\Phi(k)\Phi(\ell)tL_1\;.
  \end{align*}
  It follows that $tL_1=\Phi(\ell)tL_1$ and so $\Phi(\ell)\in
  tL_1t^{-1}$, which completes the proof.
\end{proof}

Although the proof Lemma~\ref{lem_affine} was a key point in
\cite{quasi}, it will not be explicitly needed in this paper.
However, it clearly goes hand-in-hand with Lemma~\ref{lem_extends},
which will be used crucially in the next section.  The statement and
proof of Lemma~\ref{lem_extends} were easily adapted from
\cite[Proposition~7.2]{furman}.

\begin{lem}
  \label{lem_extends}
  For $i=0,1$, let $\Gamma_i\actson K_i/L_i$ be as in
  Lemma~\ref{lem_affine}.  Suppose additionally that the action
  $K_1\actson K_1/L_1$ has trivial kernel.  Suppose that
  $\phi\from\Gamma_0\into\Gamma_1$ is a surjective homomorphism and
  that
  \[(\phi,f)\from\Gamma_0\actson K_0/L_0\longrightarrow\Gamma_1\actson K_1/L_1
  \]
  is a homomorphism of permutation groups.  Then $\phi$ extends to a
  homomorphism $\Phi\from K_0\into K_1$.
\end{lem}

\begin{proof}
  We first observe that $f$ is measure-preserving.  Indeed, letting
  $\mu_i$ denote the Haar measure on $K_i/L_i$, since $\phi$ is
  surjective we have that $f_*\mu_0$ is $\Gamma_1$-invariant.  Now, it
  is well-known that since $\Gamma_1$ is a dense subgroup of $K_1$, we
  not only have that $\Gamma_1\actson K_1/L_1$ is ergodic but also
  that it is \emph{uniquely ergodic}.  Here, an action $\Lambda\actson
  Y$ is said to be uniquely ergodic iff there exists a unique
  $\Lambda$-invariant probability measure on $Y$.  Clearly, it follows
  from this property that $f_*\mu_0=\mu_1$, so that $f$ is
  measure-preserving.

  Now, let $\nu$ be the lift of $\mu_0$ to the measure on
  $K_0/L_0\times K_1/L_1$ concentrating on the graph of $f$.  In other
  words, for $A\subset K_0/L_0\times K_1/L_1$, let
  \[\nu(A)\defeq\mu_0\set{x\in K_0/L_0\mid(x,f(x))\in A}\;.
  \]
  Next, we let
  \[R\defeq\set{(k_0,k_1)\in K_0\times K_1\mid(k_0,k_1)_*\nu=\nu}\;.
  \]
  It is easy to see that $R$ is a closed (and hence compact) subgroup
  of $K_0\times K_1$.  Moreover, it follows from the fact that
  $(\phi,f)$ is a homomorphism of permutation groups that $R$ contains
  the graph of $\phi$.  Hence, by the density of $\Gamma_i$ in $K_i$,
  we have $\pi_i(R)=K_i$, where $\pi_i$ is the canonical projection
  onto $K_i$.

  Now, consider the normal subgroups
  \begin{align*}
    R_0&\defeq\set{k_0\in K_0\mid(k_0,e)\in R}\triangleleft K_0\;,
      \textrm{ and}\\
    R_1&\defeq\set{k_1\in K_1\mid(e,k_1)\in R}\triangleleft K_1\;.
  \end{align*}
  Then loosely speaking, $R_1$ measures how far $R$ is from being the
  graph of a function.  And if $R$ is the graph of a function, then
  $R_0$ is the kernel of that function.
  
  \begin{claim*}
    $R_1=1$, and thus $R$ is the graph of a function.
  \end{claim*}
  \begin{claimproof}
    Let $k_1\in R_1$ be arbitrary, so that $(e,k_1)_*\nu=\nu$.  This
    means that for all measurable sets $A\subset K_0/L_0\times
    K_1/L_1$, we have that $\nu((e,k_1)^{-1}A)=\nu(A)$.  Appealing to
    the definition of $\nu$, we have that
    \[\mu_0\set{x\in K_0/L_0\mid(x,k_1f(x))\in A}
    = \mu_0\set{x\in K_0/L_0\mid(x,f(x))\in A}\;.
    \]
    Applying this in the case that $A$ is the graph of $f$, it follows
    that
    \[\mu_0\set{x\in K_0/L_0\mid k_1f(x)=f(x)}=1\;.
    \]
    Since $f$ is measure-preserving, we can conclude that
    \[\mu_1\set{y\in K_1/L_1\mid k_1y=y}=1\;.
    \]

    We have shown that $k_1$ fixes almost every point of $K_1/L_1$,
    and hence that $k_1\in L_1$.  Thus, we have that $R_1$ is a normal
    subgroup of $K_1$ which is contained in $L_1$.  It follows that
    $R_1$ is contained in the kernel of the action of $K_1$ on
    $K_i/L_i$, which we have assumed is trivial.
  \end{claimproof}
  Hence, $R$ is the graph of a homomorphism $\Phi\from K_0\into K_1$,
  and since $R$ contains the graph of $\phi$, we have that $\Phi$
  extends $\phi$.
\end{proof}

Let us make some further observations that help to explain the
hypotheses of the last result, and which will be useful later on when
we apply it.

\begin{rem}
  If in Lemma~\ref{lem_extends} we add the symmetric hypotheses that
  $\phi$ is injective and that $K_0\actson K_0/L_0$ has trivial
  kernel, then we can repeat the argument given in the Claim to show
  that $R_0=1$ and thus that $\Phi$ is injective.
\end{rem}

\begin{rem}
  When we apply Lemma~\ref{lem_extends}, we will unfortunately be
  interested in the case when $\phi$ is \emph{not} surjective.  To
  deal with this, consider the action of just $\phi(\Gamma_0)$ on
  $K_1/L_1$.  Since the map $x\mapsto\overline{\phi(\Gamma_0)}f(x)$ is
  $\Gamma_0$-invariant, we can use the ergodicity of $\Gamma_0\actson
  K_0/L_0$ to suppose that $f(X)$ is contained in some
  $\overline{\phi(\Gamma_0)}$-orbit, say $\overline{\phi(\Gamma_0)}z$.
  Now, $\overline{\phi(\Gamma_0)}z$ is naturally a homogeneous space
  for $\overline{\phi(\Gamma_0)}$, and we may replace $\Gamma_1\actson
  K_1/L_1$ with the action
  $\phi(\Gamma_0)\actson\overline{\phi(\Gamma_0)}z$.  We may then
  apply Lemma~\ref{lem_extends} to the latter action.
\end{rem}

\section{Superrigidity and Grassmann spaces}

The first goal of this section is to state a version of a
superrigidity theorem from ergodic theory due to Adrian Ioana.  The
conclusion of this theorem is slightly technical, and so we'll start
with the necessary definitions.  Afterwords, we shall use Ioana's
theorem to establish a template Borel incomparability result for the
actions $\SL_n(\ZZ)\actson\SL_n(\ZZ_p)$ as $n$ and $p$ vary.  We
conclude the section with Theorem~\ref{thm_glnq}, which is the key
result of the paper.  Theorem~\ref{thm_glnq} provides a strong form of
Borel incomparability for actions of $\GL_n(\QQ)$ on the $p$-adic
Grassmann spaces.

Suppose that $X$ is a standard Borel space, $\mu$ is a Borel
probability measure on $X$, and $\Gamma\actson X$ is ergodic with
respect to $\mu$.  If $\Lambda<\Gamma$ is an arbitrary subgroup, then
of course $\Lambda$ need not act ergodically on $X$.  However, if
$\Lambda$ is a subgroup of finite index in $\Gamma$, then it is not
difficult to see that there exists a $\Lambda$-invariant subset
$Z\subset X$ of positive measure such that $\Lambda\actson Z$ is
ergodic with respect to the restriction (and renormalization) of $\mu$
to $Z$.  Generally, we say that $\Lambda\actson Z$ is an \emph{ergodic
  component} for $\Gamma\actson X$ iff $\Lambda\leq\Gamma$ is a
subgroup of finite index, $Z\subset X$ is a $\Lambda$-invariant subset
of positive measure, and $\Lambda\actson Z$ is ergodic.

For example, suppose that the ergodic action $\Gamma\actson X$ has a
finite factor, that is, a finite $\Gamma$-space $X_0$ together with a
$\Gamma$-invariant and measure-preserving function $\pi\from X\into
X_0$.  Then the stabilizer $\Lambda_0$ in $\Gamma$ of any $x_0\in X_0$
is a subgroup of $\Gamma$ of finite index, and it is easy to check
that $\Lambda_0\actson\pi^{-1}(x_0)$ is an ergodic component for
$\Gamma\actson X$.

Ioana's theorem is about profinite group actions; these actions are
built up from their finite factors, and hence have a rich structure of
ergodic components.  Somewhat more precisely, if $\Gamma\actson X$ is
a probability measure-preserving action, then we say that
$\Gamma\actson X$ is \emph{profinite} iff as a $\Gamma$-space, $X$ is
the inverse limit of a directed system of finite measure-preserving
$\Gamma$-spaces.  For example, the action
$\SL_n(\ZZ)\actson\SL_n(\ZZ_p)$ is profinite; in this case
$\SL_n(\ZZ_p)$ is the inverse limit of the sequence of
$\SL_n(\ZZ)$-spaces given by $\SL_n(\ZZ/p^k\ZZ)$.  Similarly, since
$\Gr_k(\QQ_p^n)$ is a transitive $\SL_n(\ZZ_p)$-space, it is not hard
to see that $\Gr_k(\QQ_p^n)$ also carries the structure of a profinite
$\SL_n(\ZZ)$-space.  (In general, if $\Gamma\actson K$ is the inverse
limit of $\Gamma\actson K/K_n$ and $K$ acts transitively on $X$, then
$\Gamma\actson X$ is the inverse limit of $\Gamma\actson X_n$, where
$X_n$ is the set of $K_n$ orbits on $X$.)

The consequence of Ioana's superrigidity theorem which we will state
will give conditions under which any Borel homomorphism
\[f\from\Gamma\actson X\longrightarrow\Lambda\actson Y
\]
comes from a homomorphism of permutation groups
\[(\phi,f)\from\Gamma\actson X\longrightarrow\Lambda\actson Y\;.
\]
Here if $E,F$ are equivalence relations on $X,Y$, then a function
$f\from X\into Y$ is called a \emph{Borel homomorphism} from $E$ to
$F$ iff for all $x,x'\in X$,
\[x\mathrel{E}x'\implies f(x)\mathrel{F}f(x')\;.
\]
We have abused notation, so that any reference to Borel homomorphism
between actions will always refer to a Borel homomorphism between
their corresponding orbit equivalence relations.

We must remark that Ioana's theorem makes use of property (T), which
we shall not define.  It is sufficient for our purposes to note that
$\SL_n(\ZZ)$ has property (T) for $n\geq3$.  See \cite{lubotzky} for
the definition as well as a discussion of this key property.

\begin{thm}[\protect{\cite[Theorem~4.1]{ioana2}}]
  \label{thm_ioana}
  Suppose that $\Gamma$ is a countable discrete group with
  property~(T), and let $\Gamma\actson X$ be a free, ergodic and
  profinite action.  Let $\Lambda$ be a countable group,
  $\Lambda\actson Y$ a free action, and suppose that $f$ is a Borel
  homomorphism from $E_\Gamma$ to $E_\Lambda$.  Then there exists an
  ergodic component $\Gamma_0\actson X_0$ for $\Gamma\actson X$ and a
  homomorphism of permutation groups
  \[(\phi,f')\from\Gamma_0\actson X_0\longrightarrow\Lambda\actson Y
  \]
  such that for all $x\in X_0$, we have that
  $f'(x)\mathrel{E}_\Lambda f(x)$.
\end{thm}

In other words, under the hypotheses of Ioana's Theorem, the Borel
homomorphism $f$ can be replaced by one which is more or less
equivalent to $f$, and which moreover comes from a homomorphism of
permutation groups.  Ioana's theorem is stated in a significantly
higher generality in \cite{ioana}; for a proof of
Theorem~\ref{thm_ioana} from his result, see
\cite[Corollary~3.3]{quasi}.  We presently combine
Theorem~\ref{thm_ioana} together with Lemma~\ref{lem_extends} to
obtain the following result.  Although the statement of
Theorem~\ref{thm_slnz} will not be needed later on, the argument will
be expanded upon during the proof of Theorem~\ref{thm_glnq}.

\begin{thm}
  \label{thm_slnz}
  Suppose that $m,n\geq3$ are natural numbers, and $p,q$ are distinct
  primes.  Then the orbit equivalence relation induced by the action
  $\PSL_m(\ZZ)\actson\PSL_m(\ZZ_p)$ is Borel incomparable with that
  induced by $\PSL_n(\ZZ)\actson\PSL_n(\ZZ_q)$.
\end{thm}

The fact that the orbit equivalence relation induced by
$\PSL_m(\ZZ)\actson\PSL_m(\ZZ_p)$ is not Borel reducible to that
induced by $\PSL_n(\ZZ)\actson\PSL_n(\ZZ_q)$ was essentially
established by Thomas for $n<m$ in \cite[Theorem~2.4]{torsionfree} and
for $n=m$ in \cite{super}.  The arguments in this section are almost
entirely built upon his, but also apply in the case that $m<n$.

\begin{proof}
  Suppose, towards a contradiction, that $f$ is a Borel reduction from
  $\SL_m(\ZZ)\actson\SL_m(\ZZ_p)$ to $\SL_n(\ZZ)\actson\SL_n(\ZZ_q)$.
  Then the hypotheses of Theorem~\ref{thm_ioana} are satisfied, so
  there exists an ergodic component $\Gamma\actson X$ for the action
  $\SL_m(\ZZ)\actson\SL_m(\ZZ_p)$ and a homomorphism
  $\phi\from\Gamma\into\SL_n(\ZZ)$ such that
  \[(\phi,f)\from\Gamma\actson
  X\longrightarrow\SL_n(\ZZ)\actson\SL_n(\ZZ_q)
  \]
  is a permutation group homomorphism.  We now wish to apply
  Lemma~\ref{lem_extends}, but at the moment the hypothesis that
  $\phi$ is surjective isn't satisfied.  However, recall that by the
  remarks following Lemma~\ref{lem_extends}, we can suppose that
  $f(X)$ is contained in some $\phi(\Gamma)$-orbit, say
  $\overline{\phi(\Gamma)}z$.  We may now apply
  Lemma~\ref{lem_extends} to the permutation group homomorphism
  \[(\phi,f)\from
  \Gamma\actson X\longrightarrow\phi(\Gamma)\actson\overline{\phi(\Gamma)}z
  \]
  to conclude that $\phi$ lifts to a homomorphism
  $\Phi\from\bar\Gamma\into\SL_n(\ZZ_q)$, where $\bar\Gamma$ denotes
  the closure of $\Gamma$ in $\SL_m(\ZZ_p)$.  It will now suffice to
  argue that $\Phi$ is injective, for this clearly contradicts
  Proposition \ref{prop_slmzp}, below.

  Indeed, if $\Phi$ is not injective, then by Margulis's theorem on
  normal subgroups \cite[Theorem~8.1.2]{zimmer}, either $\ker(\Phi)$
  lies in the center of $\SL_n(\ZZ_p)$ or it has finite index in
  $\bar\Gamma$.  In the case when $\ker(\Phi)$ is central, $\Phi$
  clearly induces an injective homomorphism
  $\bar\Gamma'\rightarrow\PSL_n(\ZZ_q)$, where $\bar\Gamma'$ denotes
  the image of $\bar\Gamma$ in $\PSL_m(\ZZ_p)$.  Once again, this
  clearly contradicts Proposition \ref{prop_slmzp}.  Hence, we may
  suppose that $\Phi(\bar\Gamma)$ is a finite subgroup of
  $\SL_m(\ZZ_q)$.  Now, replacing $\Gamma\actson X$ with an ergodic
  subcomponent if necessary, we can suppose without loss of generality
  that $\Phi=1$.  This implies that $f$ is $\Gamma$-invariant and
  since $\bar\Gamma\actson X$ is ergodic, $f$ is almost constant.
  Hence, in this case $f$ maps a conull set into a single
  $\SL_n(\ZZ)$-orbit, which is impossible since $f$ is
  countable-to-one.
\end{proof}

\begin{prop}
  \label{prop_slmzp}
  Let $m,n\geq2$ be arbitrary and $p,q$ be distinct primes.  Then for
  any subgroup $K\leq\SL_m(\ZZ_p)$ of finite index, $K$ does not embed
  into $\SL_n(\ZZ_q)$.  Similarly, any subgroup $K\leq\PSL_m(\ZZ_p)$
  of finite index does not embed into $\PSL_n(\ZZ_q)$.
\end{prop}

For the proof, recall that $\SL_m(\ZZ_p)$ is the inverse limit of the
system of maps
\[\mathsf{pr}_k\from\SL_m(\ZZ_p)\into\SL_m(\ZZ/p^k\ZZ)\;,
\]
where $\mathsf{pr}_k$ always stands for the natural surjection.  We
shall also use the fact that any subgroup of $\SL_m(\ZZ_p)$ of finite
index contains some \emph{principle congruence subgroup}, that is, a
subgroup of the form $\ker(\mathsf{pr}_k)$.  (This is not as difficult
as some instances of the congruence subgroup problem.  Rather, it
follows from elementary properties of profinite and pro-$p$ groups.
See \cite{wilson} for the general properties of profinite groups, and
\cite[Exercise~1.9]{pro-p} for this particular fact.)

\begin{proof}[Proof of Proposition~\ref{prop_slmzp}]
  Passing to a finite index subgroup of $K$ if necessary, we may
  suppose without loss of generality that $K$ is a principle
  congruence subgroup of $\SL_m(\ZZ_p)$.  We shall use the well-known
  fact that for all $k$ there exists an $i$ such that the size of
  $\SL_m(\ZZ/p^k\ZZ)$ divides $bp^i$, where $b$ is some constant
  depending only on $m$ and $p$.  It follows that if $K'$ is any
  principle congruence subgroup of $K$ then $[K:K']$ also divides some
  $bp^i$.  The same reasoning applies to $\SL_n(\ZZ_q)$, and so there
  exists some $c\in\NN$ with the analogous properties.

  Now, suppose towards a contradiction that
  $\Phi\from K\into\SL_n(\ZZ_q)$ is an injective homomorphism.  For
  each $k$, let $N_k\leq K$ denote the kernel of the composition:
  \[K\stackrel{\Phi}{\longrightarrow}\SL_n(\ZZ_q)
  \stackrel{\mathsf{pr_k}}{\longrightarrow}\SL_n(\ZZ/q^k\ZZ)\;.
  \]
  Then for each $k$, we have that $K/N_k$ embeds into
  $\SL_n(\ZZ/q^k\ZZ)$, and so there exists a $j$ such that $[K:N_k]$
  divides $cq^j$.  On the other hand, $N_k$ also contains a principle
  congruence subgroup, and so there exists an $i$ such that $[K:N_k]$
  divides $bp^i$.  Now each $[K:N_k]$ divides both some $cq^j$ and
  some $bp^i$, and it follows that the sequence of indices $[K:N_k]$
  must be bounded.

  Now, to reach a contradiction, we shall argue that $\bigcap N_k=1$
  and hence $[K:N_k]$ tends to infinity.  Indeed, if $\gamma\in\bigcap
  N_k$ then $\gamma\in\ker(\mathsf{pr}_k\circ\Phi)$ for all $k$.
  Since $\Phi$ is injective, $\gamma\in\ker(\mathsf{pr}_k)$ for all
  $k$.  Since $\SL_m(\ZZ_q)$ is precisely the inverse limit
  corresponding to the maps $\mathsf{pr}_k$, it follows that
  $\gamma=1$, which completes the proof.
\end{proof}

\begin{rem}
  The same argument can be used to show that $\SL_n(\ZZ_p)$ does not
  even embed into any quotient of a closed subgroup of $\SL_m(\ZZ_q)$.
  To see this, one may check that such a group can again be
  expressed as an inverse limit of groups whose cardinalities are
  essentially powers of $q$ (that is, dividing $cq^i$ for some fixed
  $c$).  This is precisely the property that was required in the
  proof.
\end{rem}

Next, we shall adapt the argument of Proposition~\ref{prop_slmzp} to
establish our key result.  In order to express the result in the
greatest generality, we will use the following strengthening of the
notion ergodicity.

\begin{defn}
  Let $\Gamma\actson X$ be a probability measure-preserving action,
  and let $F$ be an arbitrary equivalence relation on the standard
  Borel space $Y$.  Then $\Gamma\actson X$ is said to be
  \emph{$F$-ergodic} iff whenever $f\from X\into Y$ is a Borel
  homomorphism from $E_\Gamma$ to $F$, there exists a conull $A\subset
  X$ such that $f(X)$ is contained in a single $F$-class.
\end{defn}

Recall that if $\Gamma\actson X$ is ergodic, then $E_\Gamma$ is
nonsmooth.  We have similarly that if $\Gamma\actson X$ is
$F$-ergodic, then $E_\Gamma\not\leq_BF$.  Moreover, in this case, if
$E$ is any countable Borel equivalence relation such that
$E_\Gamma\subset E$, then also $E\not\leq_BF$.

\begin{thm}
  \label{thm_glnq}
  Suppose that $m,n\geq3$ and $k\leq n$, and that $p,q$ are distinct
  primes.  Then $\SL_m(\ZZ)\actson\PP(\QQ_p^m)$ is $F$-ergodic, where
  $F$ is the orbit equivalence relation induced by the action
  $\GL_n(\QQ)\actson\Gr_k(\QQ_q^n)$.
\end{thm}

Once again, this has already been established by Thomas for $n<m$ in
\cite[Theorem~2.4]{torsionfree} and for $n=m$ in
\cite[Theorem~4.7]{plocal}.

\begin{proof}
  Suppose that $f$ is a Borel homomorphism from the orbit equivalence
  relation induced by $\SL_m(\ZZ)\actson\PP(\QQ_p^m)$ to that induced
  by $\SL_m(\ZZ)\actson(\PP\QQ_p^m)$.  We cannot immediately apply
  Theorem~\ref{thm_ioana}, since neither action is free.  By
  \cite[Lemma~6.2]{super}, the action $\PSL_m(\ZZ)\actson\PP(\QQ_p^m)$
  is \emph{almost free}, meaning that there exists a conull subset of
  $\PP(\QQ_p^m)$ on which $\PSL_m(\ZZ)$ acts freely.  Hence, we may
  restrict $f$ to this set to satisfy the freeness condition on the
  left-hand side.  On the other hand, for the right-hand side we must
  consider the \emph{free part}:
  \[Y\defeq\set{y\in\Gr_k(\QQ_q^n)\mid1\neq\gamma\in\PGL_n(\QQ)\implies\gamma y\neq y}\;.
  \]
  If there exists a conull subset $X\subset\PP(\QQ_p^m)$ such that
  $f(X)\subset Y$, then we may apply Theorem~\ref{thm_ioana} and we
  are done after repeating the argument from Theorem~\ref{thm_slnz}.
  Hence, since the action $\SL_m(\ZZ)\actson\PP(\QQ_p^m)$ is ergodic,
  we may suppose instead that there exists an invariant conull subset
  $X\subset\PP(\QQ_p^m)$ such that $f(X)\subset\Gr_k(\QQ_q^n)\setminus
  Y$.

  In this case, we will follow the argument found in
  \cite[Lemma~5.1]{plocal} to replace the target action
  $\PGL_n(\QQ)\actson \Gr_k(\QQ_q^n)$ with a closely related free
  action.  For this argument, it is helpful to think of elements of
  $\Gr_k(\QQ_q^n)$ as one-dimensional subspaces of the exterior power
  $\bigwedge^k\QQ_q^n$.  Here, if $y\in\Gr_k(\QQ_q^n)$ is a
  $k$-dimensional subspace of $\QQ_q^n$ with basis $v_1,\ldots,v_k$,
  then we identify $V$ with the linear subspace of
  $\bigwedge^k\QQ_q^n$ spanned by the simple tensor
  $v_1\wedge\cdots\wedge v_k$.  The relations which hold in the
  exterior algebra then ensure that this identification is
  well-defined.
  
  Now, for each $x\in X$, since $f(x)\notin Y$, we must have
  that $f(x)$ is contained in a proper eigenspace of some element of
  $\GL_n(\QQ)$.  Notice that such eigenspaces are
  \emph{$\bar\QQ$\nobreakdash-subspaces} of $\bigwedge^k\QQ_q^n$,
  where $E$ is said to be a $\bar\QQ$\nobreakdash-subspace iff there
  exists a basis for $E$ which consists only of vectors over the
  algebraic closure $\bar\QQ$ of the rationals.  Hence, for $x\in X$
  we may let $E_x$ denote a minimal $\bar\QQ$\nobreakdash-subspace of
  $\bigwedge^k\QQ_q^n$ such that $f(x)\leq E_x$.  Since there are only
  countably many possibilities for $E_x$, by the ergodicity of
  $\SL_n(\ZZ)\actson X$ we may suppose that there exists a fixed
  $\bar\QQ$\nobreakdash-subspace $V$ such that $E_x=V$ for all $x\in
  X$.  Let $H$ denote the group of projective linear transformations
  induced on $V$ by elements of the setwise stabilizer
  $\PGL_n(\QQ)_{\{V\}}$ of $V$ in $\PGL_n(\QQ)$.  Then it is easily
  checked using the minimality of $V$ that $H$ acts freely on
  $\PP(V)$.

  Now, let $d$ denote the dimension of $V$ and regard $V$ as the
  vector space $\QQ_q^d$, so that $H$ corresponds to a subgroup of
  $\PGL_d(\bar\QQ\cap\QQ_q)$.  Then we may regard $f$ as a Borel
  homomorphism from the orbit equivalence relation induced by
  $\PSL_m(\ZZ)\actson\PP(\QQ_p^m)$ to that induced by
  $H\actson\PP(\QQ_p^d)$.  Since the action of $H$ on $\PP(\QQ_q^d)$
  is free, we may now apply Theorem~\ref{thm_ioana}.  Hence, we may
  suppose that there exists an ergodic component $\Gamma\actson X$ for
  $\PSL_m(\ZZ)\actson\PP(\QQ_p^m)$ and a homomorphism
  $\phi\from\Gamma\into H$ such that
  \[(\phi,f)\from\Gamma\actson X\longrightarrow H\actson\PP(\QQ_q^d)
  \]
  is a homomorphism of permutation groups.

  Now, since $\Gamma$ has property (T) (see
  \cite[Theorem~1.5]{lubotzky}), it is in particular finitely
  generated (see \cite[Proposition~1.24]{lubotzky}).  Hence,
  $\phi(\Gamma)$ is finitely generated, and it follows that $H$ is
  contained in some $\PGL_d(F)$, where $F\leq\bar\QQ\cap\QQ_q$ is a
  finite field extension of $\QQ$.  Moreover, the commutator subgroup
  $\Gamma'\defeq[\Gamma,\Gamma]$ is a finite index subgroup of
  $\Gamma$ (see \cite[Corollary 1.29]{lubotzky}).  Since
  $\PGL_d(F)/\PSL_d(F)\iso F^\times$ is abelian, we have that
  \[\phi(\Gamma')\leq [\PGL_d(F),\PGL_d(F)]\leq\PSL_d(F)\;.
  \]
  (Actually, the latter inequality is an equality.)  Hence, replacing
  $\Gamma\actson X$ with an ergodic component for the action of
  $\Gamma'$ if necessary, we may suppose without loss of generality
  that $\phi(\Gamma)\subset\PSL_d(F)$.
  
  \begin{claim*}
    We can suppose without loss of generality that
    $\phi(\Gamma)\subset\PSL_d(\mathcal O_F)$, where $\mathcal O_F$
    denotes the ring of integers of $F$.
  \end{claim*}

  \begin{claimproof}
    Recall that an element $x\in F$ lies in the ring of integers
    $\mathcal O_F$ if and only if $v(x)\geq0$ for every nonarchimedian
    valuation $v$ on $F$.  More generally, if $S$ is a set of
    valuations on $F$, then we say that $x\in F$ is an
    \emph{$S$\nobreakdash-integer} iff $v(x)\geq0$ for all
    nonarchimedian valuations $v\notin S$.  We denote the ring of
    $S$-integers of $F$ by $F(S)$, so that in particular the notation
    implies that $\mathcal O_F=F(\emptyset)$.

    Now, note that $F$ is the union of the rings $F(S)$ as $S$ varies
    over all finite sets of valuations on $F$.  Therefore, using the
    fact that $\phi(\Gamma)$ is finitely generated, there exists a
    finite set $S$ of valuations on $F$ such that
    \begin{equation}
      \label{eqn_sld1}
      \phi(\Gamma)\subset\SL_d(F(S))\;.
    \end{equation}
    Next, for any valuation $v$ on $F$, let $F_v$ denote the completion
    of $F$ with respect to $v$, and $\mathcal O_v$ the ring of
    integers of $F_v$.  It is clear from the definitions that we have
    \begin{equation}
      \label{eqn_sld2}
      \PSL_d(\mathcal O_F)
      =\PSL_d(F(S))\cap\bigcap_{v\in S}\PSL_d(\mathcal O_v)\;.
    \end{equation}
    By \cite[Theorem VII.5.16]{margulis}, for each nonarchimedian
    valuation $v$ on $F$, $\phi(\Gamma)$ is relatively compact in
    $\SL_d(F_v)$.  (To see that the hypotheses of \cite[Theorem
    VII.5.16]{margulis} are satisfied, note that by \cite[Theorem
    VIII.3.10]{margulis}, the Zariski closure in $\PSL_d(F_v)$ of
    $\phi(\Gamma)$ is semisimple.)  Since $\PSL_d(\mathcal O_v)$ is an
    open subgroup of $\PSL_d(F_v)$, we have that
    $\phi(\Gamma)\cap\PSL_d(\mathcal O_v)$ is of finite index in
    $\phi(\Gamma)$.  Since $S$ is finite, it follows that
    \[\phi(\Gamma)\cap\bigcap_{v\in S}\PSL_d(\mathcal O_v)
    \]
    is also of finite index in $\phi(\Gamma)$.  This, together with
    equations \eqref{eqn_sld1} and \eqref{eqn_sld2}, implies that
    $\phi(\Gamma)\cap\PSL_d(\mathcal O_F)$ has finite index in
    $\phi(\Gamma)$.  Thus, replacing $\Gamma$ with a subgroup of
    finite index establishes the claim.
  \end{claimproof}

  Now, recall that $F\subset\QQ_q$, and it follows that $\mathcal
  O_F\subset\ZZ_q$.  (Indeed, $F$ carries a $q$-adic valuation and so
  each $x\in\mathcal O_F$ has $v_q(x)\geq0$.)  Combining this with the
  Claim, we have that $\phi(\Gamma)\subset\PSL_d(\ZZ_q)$.  For the
  remainder of the proof, let $K_0$ denote the closure of $\Gamma$ in
  $\SL_m(\ZZ_p)$ and let $K_1$ denote the closure of $\phi(\Gamma)$ in
  $\PSL_d(\ZZ_q)$.  Roughly speaking, we now wish to maneuver into a
  situation where we can apply Lemma~\ref{lem_extends} to the
  permutation group homomorphism
  \[(\phi,f)\from
  \Gamma\actson X\longrightarrow\phi(\Gamma)\actson\PP(\QQ_q^d)\;.
  \]
  to obtain an embedding of $K_0$ into $K_1$, which would be a
  contradiction.  First, by the remarks following
  Lemma~\ref{lem_extends}, we can suppose that $f(X)$ is contained in
  a single $K_1$-orbit, say $K_1z$.  We would like to apply
  Lemma~\ref{lem_extends} to the permutation group homomorphism
  \[(\phi,f)\from
  \Gamma\actson X\longrightarrow\phi(\Gamma)\actson K_1z\;,
  \]
  but it is not necessarily the case that $K_1$ acts faithfully on
  $K_1z$.  However, if there is a kernel $N\trianglelefteq K_1$ for
  this action, then $K_1z$ is naturally a homogeneous $K_1/N$-space.
  Composing $(\phi,f)$ with the obvious factor map, we may now apply
  Lemma~\ref{lem_extends} to obtain a homomorphism $\Phi\from K_0\into
  K_1/N$.  Arguing as in the proof of Theorem~\ref{thm_slnz} we can
  suppose that $\Phi$ is injective, but this contradicts the remark
  following Proposition~\ref{prop_slmzp}.
\end{proof}

\section{Torsion-free abelian groups}

In this section, we shall use Theorem~\ref{thm_glnq} to prove
Theorems~A and B.  In order to do so, we must first show that the
isomorphism equivalence relations on spaces of local torsion-free
abelian groups are in fact very closely related to orbit equivalence
relations on Grassmann spaces.  For this, we shall rely on some
methods of Hjorth, Thomas and myself which ultimately make use of the
Kurosh-Malcev invariants for torsion-free abelian groups of finite
rank.

Recall that $\oqiso_{m,p}$ denotes the quasi-isomorphism relation on
the space of $p$-local torsion-free abelian groups of rank $m$.  The
following result is a straightforward application of the Kurosh-Malcev
$p$-adic localization technique.

\begin{lem}[\protect{\cite[Theorem 4.3]{plocal}}]
  \label{lem_qiso}
  The quasi-isomorphism relation $\oqiso_{m,p}$ is Borel bireducible
  with the orbit equivalence relation induced by the action of
  $\GL_n(\QQ)$ on the full Grassmann space $\Gr(\QQ_p^m)$ of all
  vector subspaces of $\QQ_p^m$.
\end{lem}

Of course, the full Grassmann space decomposes naturally into the
invariant components $\Gr_k(\QQ_p^m)$, for $k=0,\ldots,n$.

\begin{cor}[Theorem~A]
  \label{mainthmp1}
  If $m,n\geq3$ and $p,q$ are distinct primes, then $\oqiso_{m,p}$ is
  Borel incomparable with $\oqiso_{n,q}$.
\end{cor}

\begin{proof}
  Suppose that there exists a Borel reduction from $\oqiso_{m,p}$ to
  $\oqiso_{n,q}$.  Then by Lemma~\ref{lem_qiso} there exists a Borel
  reduction
  \[f\from\GL_m(\QQ)\actson\Gr(\QQ_p^m)
  \longrightarrow\GL_n(\QQ)\actson\Gr(\QQ_q^n)\;.
  \]
  Now, consider the restriction of $f$ to $\PP(\QQ_p^m)$.  Since each
  $Gr_k(\QQ_q^n)$ is $\GL_n(\QQ)$-invariant, by the ergodicity of
  $\SL_n(\ZZ)\actson\PP(\QQ_q^n)$, we can adjust $f$ on a null set to
  suppose that $f$ takes values in $\Gr_k(\QQ_q^n)$ for some fixed
  $k$.  Therefore, $f$ is a Borel homomorphism
  \[f\from\SL_m(\ZZ)\actson\PP(\QQ_p^m)\longrightarrow
  \GL_n(\QQ)\actson\Gr_k(\QQ_q^n)\;.
  \]
  By Theorem \ref{thm_glnq}, the image $f\!\left(\PP(\QQ_p^m)\right)$
  is a countable set, which is impossible since $f$ is a
  countable-to-one function.
\end{proof}

The proof of Theorem~B is nearly identical, modulo the following
rather technical piece of machinery.

\begin{lem}[\protect{\cite[Lemma~4.1]{quasi}}]
  \label{lem_iso}
  The isomorphism relation $\oiso_{m,p}$ is Borel bireducible with an
  equivalence relation $\oiso_{m,p}'$ which, thought of as a set of
  pairs, lies properly between the orbit equivalence relations induced
  by the actions $\SL_m(\ZZ)\actson\Gr(\QQ_p^m)$ and
  $\GL_m(\QQ)\actson\Gr(\QQ_p^m)$.
\end{lem}

\begin{cor}[Theorem~B]
  \label{mainthmp2}
  If $m,n\geq3$ and $p,q$ are distinct primes, then $\oiso_{m,p}$ is
  Borel incomparable with $\oiso_{n,q}$.
\end{cor}

\begin{proof}
  If there exists a Borel reduction from $\oiso_{m,p}$ to
  $\oiso_{n,q}$, then there exists a Borel reduction $f$ from
  $\oiso_{m,p}'$ to $\oiso_{n,q}'$.  It follows from the containments
  described in Lemma~\ref{lem_iso} that $f$ is also a Borel
  homomorphism:
  \[f\from\SL_m(\ZZ)\actson\Gr(\QQ_p^m)
  \longrightarrow\GL_n(\QQ)\actson\Gr(\QQ_q^n)\;.
  \]
  Arguing as in the proof of Corollary~\ref{mainthmp1}, we again
  arrive at a contradiction.
\end{proof}

\bibliographystyle{alpha}
\begin{singlespace}
  \bibliography{cber}

\begin{thebibliography}{DdSMS91}

\bibitem[AK00]{adamskechris}
Scot Adams and Alexander~S. Kechris.
\newblock Linear algebraic groups and countable {B}orel equivalence relations.
\newblock {\em J. Amer. Math. Soc.}, 13(4):909--943 (electronic), 2000.

\bibitem[Cos10]{quasi}
Samuel Coskey.
\newblock The classification of torsion-free abelian groups of finite rank up
  to isomorphism and up to quasi-isomorphism.
\newblock {\em {\rm Accepted to the }Transactions of the American Mathematical
  Society. {\rm See also }{\tt arXiv:0902.1218 [math.LO]}}, 2010.

\bibitem[DdSMS91]{pro-p}
J.~D. Dixon, M.~P.~F. du~Sautoy, A.~Mann, and D.~Segal.
\newblock {\em Analytic pro-{$p$}-groups}, volume 157 of {\em London
  Mathematical Society Lecture Note Series}.
\newblock Cambridge University Press, Cambridge, 1991.

\bibitem[FM77]{feldmanmoore}
Jacob Feldman and Calvin~C. Moore.
\newblock Ergodic equivalence relations, cohomology, and von {N}eumann
  algebras. {I}.
\newblock {\em Trans. Amer. Math. Soc.}, 234(2):289--324, 1977.

\bibitem[Fur05]{furman}
Alex Furman.
\newblock Outer automorphism groups of some ergodic equivalence relations.
\newblock {\em Comment. Math. Helv.}, 80(1):157--196, 2005.

\bibitem[Gef96]{gefter}
Sergey~L. Gefter.
\newblock Outer automorphism group of the ergodic equivalence relation
  generated by translations of dense subgroup of compact group on its
  homogeneous space.
\newblock {\em Publ. Res. Inst. Math. Sci.}, 32(3):517--538, 1996.

\bibitem[Hjo99]{hjorth}
Greg Hjorth.
\newblock Around nonclassifiability for countable torsion free abelian groups.
\newblock In {\em Abelian groups and modules ({D}ublin, 1998)}, Trends Math.,
  pages 269--292, Basel, 1999. Birkh\"auser.

\bibitem[Ioa07a]{ioana2}
Adrian Ioana.
\newblock Cocycle superrigidity for profinite actions of property ({T}) groups.
\newblock {\em {\rm Submitted, and available at }{\tt arXiv:0805.2998v1
  [math.GR]}}, 2007.

\bibitem[Ioa07b]{ioana}
Adrian Ioana.
\newblock {\em Some rigidity results in the orbit equivalence theory of
  non-amenable groups}.
\newblock PhD thesis, UCLA, 2007.

\bibitem[Lub03]{lubotzky}
Alex Lubotzky.
\newblock {\em On property $(\tau)$}.
\newblock In preparation, 2003.

\bibitem[Mar91]{margulis}
Gregori~A. Margulis.
\newblock {\em Discrete subgroups of semisimple {L}ie groups}.
\newblock Number~3 in Ergebnisse der mathematik und ihrer grenzgebiete.
  Springer, 1991.

\bibitem[Tho02a]{plocal}
Simon Thomas.
\newblock The classification problem for $p$-local torsion-free abelian groups
  of finite rank.
\newblock {\em Unpublished preprint}, 2002.

\bibitem[Tho02b]{thomas_survey}
Simon Thomas.
\newblock On the complexity of the classification problem for torsion-free
  abelian groups of rank two.
\newblock {\em Acta Math.}, 189(2):287--305, 2002.

\bibitem[Tho03a]{torsionfree}
Simon Thomas.
\newblock The classification problem for torsion-free abelian groups of finite
  rank.
\newblock {\em J. Amer. Math. Soc.}, 16(1):233--258 (electronic), 2003.

\bibitem[Tho03b]{super}
Simon Thomas.
\newblock Superrigidity and countable {B}orel equivalence relations.
\newblock {\em Ann. Pure Appl. Logic}, 120(1-3):237--262, 2003.

\bibitem[Wil98]{wilson}
John~S. Wilson.
\newblock {\em Profinite groups}, volume~19 of {\em London Mathematical Society
  Monographs. New Series}.
\newblock Oxford University Press, New York, 1998.

\bibitem[Zim84]{zimmer}
Robert~J. Zimmer.
\newblock {\em Ergodic theory and semisimple groups}, volume~81 of {\em
  Monographs in Mathematics}.
\newblock Birkh\"auser Verlag, Basel, 1984.

\end{thebibliography}
\end{singlespace}

\end{document}